\documentclass[final,10pt,journal,twoside]{IEEEtran}
\usepackage{microtype}
\usepackage{siunitx}
\usepackage{graphics, graphicx}
\usepackage{amsthm, amssymb}
\usepackage[cmex10]{amsmath}
\usepackage{cite}
\usepackage{array}
\usepackage{booktabs}
\usepackage{enumerate}
\usepackage[pdftex,colorlinks=true,linkcolor=blue,anchorcolor=blue,citecolor=blue]{hyperref}
\usepackage{stfloats}% 分栏插图     h放在此处   t放在顶端   b放在底端   p放在本页
\usepackage{subfigure}% 多个小图放一起
\usepackage{multirow} % 表格
\usepackage{algorithm}% 算法
\usepackage{algorithmic}
\usepackage{amsmath}
\usepackage{enumitem}
\usepackage{color} % 改变字体颜色
\usepackage{makecell}
\usepackage{dsfont} % mathds 字体
\usepackage[normalem]{ulem}
\useunder{\uline}{\ul}{}
\newtheorem{theorem}{Theorem}
\newtheorem{lemma}{Lemma}

\allowdisplaybreaks[4]
\input{tcilatex}
\begin{document}

\title{An Accelerated Proximal Gradient-based Model Predictive Control Algorithm}
\author{ Jia~Wang, Ying~Yang,~\IEEEmembership{Senior Member,~IEEE}

\thanks{%
This work is supported by the National Key R$\&$D Program of
China (No. 2021YFB3301204) and the National Natural Science
Foundation of China under grants 62173003 and U1713223. J. Wang and Y. Yang are with the State Key Lab for Turbulence and
Complex Systems, Department of Mechanics and Engineering Science, College of Engineering, Peking University, Beijing 100871, China (Email: pkuwangjia@pku.edu.cn; yy@pku.edu.cn). Corresponding author: Ying Yang.}
}
\maketitle

\begin{abstract}
In this letter, an accelerated quadratic programming (QP) algorithm is proposed based on the proximal gradient method. The algorithm can achieve convergence rate $O(1/p^{\alpha})$, where $p$ is the iteration number and $\alpha$ is the given positive integer. The proposed algorithm improves the convergence rate of existing algorithms that achieve $O(1/p^{2})$. The key idea is that iterative parameters are selected from a group of specific high order polynomial equations. The performance of the proposed algorithm is assessed on the randomly generated model predictive control (MPC) optimization problems. The experimental results show that our algorithm can outperform the state-of-the-art optimization software MOSEK and ECOS for the small size MPC problems.

\end{abstract}


\section{Introduction}

Many engineering optimization problems can be formulated to quadratic programming (QP) problems. For example, model predictive control (MPC), which has been widely used in many industrial processes \cite{Qin_2003}. However, the solving of the QP problem is often computationally demanding. In practice, many industrial processes require a fast solution of the problem, for example, the control systems with high sampling rate \cite{Juan_2014}. Therefore, it is important to develop an accelerated algorithm for solving QP problems.

For reducing the computational load of the controller, QP problems are solved by using online optimization technique. Popular QP solvers use an interior-point method \cite{Domahidi2012}, an active-set method \cite{Ferreau_2014} and a dual Newton method \cite{Frasch_2015}. However, above solvers require the solution of the linearization system of the Karush-Kuhn-Tucker (KKT) conditions at every iteration. For this reason, the great attention has been given to the first-order methods for the online optimization \cite{Parys_2019,Giselsson_2013,Alberto_2021}. In recent years, the proximal gradient-based accelerated algorithms are widely used to solve MPC problems \cite{Giselsson2014}. Specifically, the iterative algorithm is designed based on the proximal gradient method (PGM) to deal with the constraint of Lagrange multiplier more easily \cite{Giselsson_2013,Parys_2019,Giselsson2014,Giselsson2015}. Moreover, methods in \cite{Beck_2009,Nesterov2013}, i.e., fast iterative shrinkage-thresholding algorithm (FISTA) improves the iteration convergence rate from $\small O(1/p)$ to $\small O(1/p^{2})$. The key idea of this improvement is that the positive real root of a specific quadratic polynomial equation is selected as the iterative parameter. Inspired by the work in \cite{Beck_2009} and \cite{Giselsson_2013}, an accelerated PGM algorithm is proposed for fast solving QP problems in this letter. We show that the FISTA in \cite{Beck_2009} is a special case of the proposed method and the convergence rate can be improved from $\small O(1/p^{2})$ in \cite{Beck_2009} to $\small O(1/p^{\alpha})$ by selecting the positive real roots of a group of high order polynomial equations as the iterative parameters. To assess the performance of the proposed algorithm, a batch of randomly generated MPC problems are solved. Then, comparing the resulted execution time to state-of-the-art optimization softwares, in particular MOSEK \cite{Erling2003} and ECOS \cite{Domahidi2013ecos}.

The paper is organized as follows. In Section \ref{section2}, the QP problem is formulated into the dual form and the PGM is introduced. The accelerated PGM for the dual problem is proposed in Section \ref{section3}. In Section \ref{section4}, the numerical experiment based on the MPC are provided. Section \ref{section5} concludes the result of this letter.

\section{Problem Formulation}
\label{section2}

\subsection{Primal and Dual Problems}

Consider the standard quadratic programming problem

\begin{small}
\begin{equation}\label{standard_QP}
\begin{split}
&\min\limits_{\xi}
\frac{1}{2}
\xi^{T}\mathcal{H}\xi
+
\mathcal{G}^{T}\xi
\\
&\ s.t.\ \mathcal{A}\xi\leq\mathcal{B}.
\end{split}
\end{equation}
\end{small}Assume that there exists $\xi$ such that $ \mathcal{A}\xi<\mathcal{B}$, which means that the Slater's condition holds and there is no duality gap \cite{Boyd}, the dual problem of (\ref{standard_QP}) is formulated as

\begin{small}
\begin{equation}\label{dual_P}
\sup\limits_{\mu\geq0}
\inf\limits_{\xi}
\begin{bmatrix}
\frac{1}{2}\xi^{T}\mathcal{H}\xi
+
\mathcal{G}^{T}\xi
+
\mu^{T}(\mathcal{A}\xi-\mathcal{B})
\end{bmatrix}.
\end{equation}
\end{small}Take the partial derivative with respect to $\small \xi$ and according to the first-order optimality condition, we have

\begin{small}
\begin{equation*}
\begin{split}
&\frac{\partial}{\partial \xi}
\begin{bmatrix}
\frac{1}{2}\xi^{T}\mathcal{H}\xi+
(\mathcal{A}^{T}\mu+\mathcal{G})^{T}\xi-
\mu^{T}\mathcal{B}
\end{bmatrix}=0 \\
&\Rightarrow \xi=\mathcal{H}^{-1}(-\mathcal{A}^{T}\mu-\mathcal{G}).
\end{split}
\end{equation*}
\end{small}In this way, (\ref{dual_P}) is transformed into

\begin{small}
\begin{equation}\label{dual_problem}
\sup\limits_{\mu\geq0}
\begin{bmatrix}
-\frac{1}{2}
(\mathcal{A}^{T}\mu+\mathcal{G})^{T}
\mathcal{H}^{-1}
(\mathcal{A}^{T}\mu+\mathcal{G})
-
\mathcal{B}^{T}\mu
\end{bmatrix}.
\end{equation}
\end{small}Let $\small f(\mu)=\frac{1}{2}(\mathcal{A}^{T}\mu+\mathcal{G})^{T}
\mathcal{H}^{-1}(\mathcal{A}^{T}\mu+\mathcal{G})+\mathcal{B}^{T}\mu$ be the new objective, then minimizing $\small f(\mu)$ yields the new optimization problem.

\subsection{Proximal Gradient Method}

In this subsection, the PGM is used to solve the dual problem. Specifically, the following nonsmooth function $\small g$ is introduced to describe the constraint of $\small f(\mu)$

\begin{small}
\begin{equation}
g(\mu)
=
\begin{cases}
0, & \mbox{if } \mu\geq0 \\
+\infty, & \mbox{otherwise}.
\end{cases}
\end{equation}
\end{small}In this way, the constrained optimization problem $\small \min\limits_{\mu\geq0}f(\mu)$ is equivalent to the unconstrained one, i.e., $\small \min\limits_{\mu}f(\mu)+g(\mu)$. Based on the work in \cite{Beck_2009}, let $\small \zeta^{p}=\mu^{p}+\frac{\tau_{p}-1}{\tau_{p+1}}(\mu^{p}-\mu^{p-1})$, where $\small \tau_{p}>0$ for $\small p=1,2,\cdots$ and $\small p$ is iteration number. Then the above problem can be solved by

\begin{small}
\begin{equation}\label{proximal_step}
\mu^{p+1}
=
\mathrm{P}_{\mu}
(\zeta^{p}-\frac{1}{L}\nabla f(\zeta^{p})),
\end{equation}
\end{small}where $\small L$ is the Lipschitz constant of $\small \nabla f$ and $\small \mathrm{P}_{\mu}$ is the Euclidean projection to $\small \{\mu|\mu\geq0\}$. According to the result in \cite{Giselsson_2013}, there is
$\small \nabla f(\mu)=\mathcal{A}\mathcal{H}^{-1}(\mathcal{A}^{T}\mu+\mathcal{G})+\mathcal{B}$, then we have

\begin{small}
\begin{equation}
\nabla f(\zeta^{p})
=
-
\mathcal{A}
\begin{bmatrix}
\xi^{p}
+
\frac{\tau_{p}-1}{\tau_{p+1}}(\xi^{p}-\xi^{p-1})
\end{bmatrix}+\mathcal{B}.
\end{equation}
\end{small}Therefore, (\ref{proximal_step}) can be written as

\begin{small}
\begin{equation}
\mu_{l}^{p+1}
=
\max
\begin{Bmatrix}
0,
\zeta_{l}^{p}+\frac{1}{L}
\begin{bmatrix}
\mathcal{A}_{l}
(
\xi^{p}
+
\frac{\tau_{p}-1}{\tau_{p+1}}(\xi^{p}-\xi^{p-1})
)
-\mathcal{B}_{l}
\end{bmatrix}
\end{Bmatrix}
\end{equation}
\end{small}where $\mu_{l}$ denotes the $l$-th component of the vector $\mu$. $\mathcal{A}_{l}$ and $\mathcal{B}_{l}$ are the $\small l$-th row of $\mathcal{A}$ and $\mathcal{B}$. The classical PGM to solve $\min\limits_{\mu}f(\mu)+g(\mu)$ can be summarized as Algorithm \ref{algorithm1}, in which $\tau_{p}$ and $\tau_{p+1}$ are iterative parameters, which will be discussed in the next subsection.

\begin{algorithm}
\caption{Proximal Gradient Method.}
\label{algorithm1}
\begin{algorithmic}[1]
\REQUIRE ~~\\
Initial parameters $\small \zeta_{l}^{1}=\mu_{l}^{0}$, $\small \tau_{1}=1$ and $\small \bar{\xi}^{1}=\xi_{0}$.
\ENSURE ~~The optimal decision variable $\small \mu^{*}$. \\
\WHILE {$\small p\geq1$}
\STATE $\small \mu_{l}^{p}=\max\{0,\zeta_{l}^{p}+\frac{1}{L}(\mathcal{A}_{l}\bar{\xi}^{p}-\mathcal{B}_{l})\},\ \forall l$.\\
\STATE Looking up table for $\tau_{p}$ and $\tau_{p+1}$.\\
\STATE $\small \xi^{p}=\mathcal{H}^{-1}(-\mathcal{A}^{T}\mu^{p}-\mathcal{G})$.\\
\STATE $\small \zeta^{p+1}=\mu^{p}+\frac{\tau_{p}-1}{\tau_{p+1}}(\mu^{p}-\mu^{p-1})$.\\
\STATE $\small \bar{\xi}^{p+1}=\xi^{p}+\frac{\tau_{p}-1}{\tau_{p+1}}(\xi^{p}-\xi^{p-1})$.\\
\STATE $\small p=p+1.$
\ENDWHILE
\end{algorithmic}
\end{algorithm}

\section{Accelerated MPC Iteration}

\label{section3}

\subsection{Accelerated Scheme and Convergence Analysis}

The traditional iterative parameters $\tau_{p}$ and $\tau_{p+1}$ are selected based on the positive real root of the following second-order polynomial equation

\begin{small}
\begin{equation}\label{2_polynomial}
\tau_{p+1}^{2}-\tau_{p+1}-\tau_{p}^{2}=0
\end{equation}
\end{small}with $\tau_{1}=1$. With the aid of (\ref{2_polynomial}), the convergence rate $O(1/p^{2})$ can be achieved \cite{Beck_2009}. In this work, we show that the convergence rate can be enhanced to $O(1/p^{\alpha})$ only by selecting iterative parameters appropriately. Specifically, for the given order $\alpha\in\{2,3,\cdots\}$, iterative parameters are determined by the positive real root of the $\alpha$th-order equation

\begin{small}
\begin{equation}\label{alpha_polynomial}
\tau_{p+1}^{\alpha}-\tau_{p+1}^{\alpha-1}-\tau_{p}^{\alpha}=0
\end{equation}
\end{small}with the initial value $\tau_{1}=1$, instead of (\ref{2_polynomial}). This is the main difference between our method and the method in \cite{Beck_2009}.

\begin{lemma}\label{lemma1}
The $\alpha$th-order polynomial equation (\ref{alpha_polynomial}) has following properties:
\begin{enumerate}
  \item For $\alpha\in\{2,3,\cdots\}$, the polynomial equation (\ref{alpha_polynomial}) has the unique positive real root.
  \item For $p\in\{1,2,\cdots\}$, the unique positive real root has the lower bound as
      \begin{small}
      \begin{equation*}
      \tau_{p}\geq\frac{p+\alpha-1}{\alpha}.
      \end{equation*}
      \end{small}
\end{enumerate}
\end{lemma}

\begin{proof}
For the first argument, we first show that $\tau_{p}>0$ for all positive integer $p\geq1$ with the aid of mathematical induction. Specifically, the base case $\tau_{1}>0$ holds since the given initial value $\tau_{1}=1$. Assume the induction hypothesis that $\tau_{p}>0$ holds. Then we have

\begin{small}
\begin{equation}\label{temp_1}
\tau_{p+1}^{\alpha}-\tau_{p+1}^{\alpha-1}-\tau_{p}^{\alpha}=0
\Rightarrow
\tau_{p+1}^{\alpha-1}(\tau_{p+1}-1)>0.
\end{equation}
\end{small}Since (\ref{temp_1}) holds for all $\alpha\in\{2,3,\cdots\}$, $\tau_{p+1}$ should be a positive value greater than one, therefore we have $\tau_{p+1}>0$. In this way, we conclude that $\tau_{p}>0$ for all $p\in\{1,2,\cdots\}$. To the uniqueness of positive real root, let

\begin{small}
\begin{equation}
f_{1}(\tau_{p+1})\triangleq\tau_{p+1}^{\alpha}-\tau_{p+1}^{\alpha-1}-\tau_{p}^{\alpha},
\end{equation}
\end{small}which has the derivative as

\begin{small}
\begin{equation}
f_{1}^{'}(\tau_{p+1})=\tau_{p+1}^{\alpha-2}(\alpha\tau_{p+1}-\alpha+1),
\end{equation}
\end{small}it has zero points $\tau_{p+1}=0$ and $\tau_{p+1}=\frac{\alpha-1}{\alpha}$. Therefore, $f_{1}(\tau_{p+1})$ monotonically decreases from $f_{1}(0)$ to $f_{1}(\frac{\alpha-1}{\alpha})$, and monotonically increases from $f_{1}(\frac{\alpha-1}{\alpha})$ to $f_{1}(+\infty)$. Since $f_{1}(0)=-\tau_{p}^{\alpha}<0$ and $\lim\limits_{\tau_{p+1}\rightarrow+\infty}f_{1}(\tau_{p+1})=+\infty$, the function $f_{1}(\tau_{p+1})$ has only one zero point, which implies that the equation (\ref{alpha_polynomial}) has the unique positive real root.

For the second argument, we still use the mathematical induction. The base case $\tau_{1}\geq1$ holds since the given initial value $\tau_{1}=1$. Assume the induction hypothesis that $\tau_{p}\geq\frac{p+\alpha-1}{\alpha}$ holds. To show $\tau_{p+1}\geq\frac{p+\alpha}{\alpha}$, we can equivalently prove that the inequality $f_{1}(\frac{p+\alpha}{\alpha})<0$ holds. Moreover, since the induction hypothesis $\tau_{p}\geq\frac{p+\alpha-1}{\alpha}>0$, we can prove the following inequality

\begin{small}
\begin{equation}
\begin{pmatrix}
\frac{p+\alpha}{\alpha}
\end{pmatrix}^{\alpha}
-
\begin{pmatrix}
\frac{p+\alpha}{\alpha}
\end{pmatrix}^{\alpha-1}
-
\begin{pmatrix}
\frac{p+\alpha-1}{\alpha}
\end{pmatrix}^{\alpha}
<0
\end{equation}
\end{small}holds and it is equivalent to show

\begin{small}
\begin{equation}
f_{2}(p)
\triangleq
(\alpha-1)\ln
\begin{pmatrix}
\frac{p+\alpha}{\alpha}
\end{pmatrix}
+
\ln
\begin{pmatrix}
\frac{p}{\alpha}
\end{pmatrix}
-
\alpha\ln\begin{pmatrix}
\frac{p+\alpha-1}{\alpha}
\end{pmatrix}
<0
\end{equation}
\end{small}holds for all $p\in\{1,2,\cdots\}$. The derivative of $f_{2}(p)$ is

\begin{small}
\begin{equation}
f_{2}^{'}(p)
=
\frac{\alpha(\alpha-1)}{p(p+\alpha)(p+\alpha-1)},
\end{equation}
\end{small}which implies that the function $f_{2}(p)$ monotonically increases from $f_{2}(1)$ to $f_{2}(+\infty)$. Next, we show that $f_{2}(1)<0$. Notice that

\begin{small}
\begin{equation}\label{f_2_1}
f_{2}(1)
=
(\alpha-1)\ln
\begin{pmatrix}
\frac{1+\alpha}{\alpha}
\end{pmatrix}
+
\ln
\begin{pmatrix}
\frac{1}{\alpha}
\end{pmatrix},
\end{equation}
\end{small}which is a function about $\alpha$, hence, denote (\ref{f_2_1}) as

\begin{small}
\begin{equation}
f_{3}(\alpha)
\triangleq
(\alpha-1)\ln(\alpha+1)
-
\alpha\ln\alpha.
\end{equation}
\end{small}The first and second derivatives of $f_{3}(\alpha)$ are

\begin{small}
\begin{subequations}
\begin{align}
f_{3}^{'}(\alpha)
&=
\ln
\begin{pmatrix}
\frac{\alpha+1}{\alpha}
\end{pmatrix}
-
\frac{2}{\alpha+1},
\\
f_{3}^{''}(\alpha)
&=
\frac{\alpha-1}{\alpha(\alpha+1)^{2}},
\end{align}
\end{subequations}
\end{small}which implies that $f_{3}(\alpha)$ monotonically
increases for $\small \alpha\in\{2,3,\cdots\}$. Since $f_{3}(2)<0$ and $\lim\limits_{\alpha\rightarrow+\infty}f_{3}(\alpha)=0$, we conclude that $f_{2}(1)<0$. Finally, according to $\lim\limits_{p\rightarrow+\infty}f_{2}(p)=0$, the unique positive real root $\tau_{p}$ has the lower bound $\frac{p+\alpha-1}{\alpha}$.
\end{proof}

%\begin{remark}
%\textcolor{red}{
%The first argument of Lemma \ref{lemma1} is to clarify that the specific group equations (\ref{alpha_polynomial}) for $p\in\{1,2,\cdots\}$
%}
%\end{remark}

For the purpose of saving computing time, the $\alpha$th-order equations are solved offline and the roots are stored in a table. In this work, the look-up table is obtained by recursively solving the polynomial equation (\ref{alpha_polynomial}) in MATLAB environment, which can be summarized as Algorithm \ref{algorithm2}. Notice that the MATLAB function $\text{roots}(\cdot)$ is used for the polynomial root seeking. The following theorem show that the convergence rate can be improved to $O(1/p^{\alpha})$ by using (\ref{alpha_polynomial}).

\begin{theorem}\label{Theorem_2}
For $\alpha\in\{2,3,\cdots\}$, let $\xi^{*}$ and $\mu^{*}$ denote the optimizers of the problems (\ref{standard_QP}) and (\ref{dual_problem}) respectively, the convergence rate of the primal variable by Algorithm \ref{algorithm1} is

\begin{small}
\begin{equation}\label{convergence_rate_general}
\|\xi^{p}-\xi^{*}\|_{2}^{2}
\leq
\frac{\alpha^{\alpha}L\|\mu^{0}-\mu^{*}\|_{2}^{2}}{\underline{\sigma}(\mathcal{H})(p+\alpha-1)^{\alpha}},\ p=1,2,\cdots
\end{equation}
\end{small}where $\small \underline{\sigma}(\cdot)$ denotes the minimum eigenvalue.

\end{theorem}

\begin{proof}
Let $\small \upsilon^{p}=f(\mu^{p})-f(\mu^{*})$, according to Lemma 2.3 in \cite{Beck_2009}, we have

\begin{small}
\begin{subequations}
\begin{align}
\begin{split}
\frac{2}{L}(\upsilon^{p}-\upsilon^{p+1})
&\geq\|\mu^{p+1}-\zeta^{p+1}\|_{2}^{2}\\
&+2\langle\mu^{p+1}-\zeta^{p+1},\zeta^{p+1}-\mu^{p}\rangle,
\end{split}\label{lemma_4_1_1} \\
\begin{split}
-\frac{2}{L}\upsilon^{p+1}
&\geq\|\mu^{p+1}-\zeta^{p+1}\|_{2}^{2}\\
&+2\langle\mu^{p+1}-\zeta^{p+1},\zeta^{p+1}-\mu^{*}\rangle.
\end{split}\label{lemma_4_1_2}
\end{align}
\end{subequations}
\end{small}Follow the line of Lemma 4.1 in \cite{Beck_2009}, multiply $\small (\tau_{p+1}-1)$ to the both sides of (\ref{lemma_4_1_1}) and add the result to (\ref{lemma_4_1_2}), which leads to

\begin{small}
\begin{equation}\label{Most_Important_Ieq_1}
\begin{split}
\frac{2}{L}
\begin{bmatrix}
(\tau_{p+1}-1)\upsilon^{p}-\tau_{p+1}\upsilon^{p+1}
\end{bmatrix}\geq
\tau_{p+1}\|\mu^{p+1}-\zeta^{p+1}\|_{2}^{2}\\
+2\langle\mu^{p+1}-\zeta^{p+1},\tau_{p+1}\zeta^{p+1}-(\tau_{p+1}-1)\mu^{p}-\mu^{*}\rangle.
\end{split}
\end{equation}
\end{small}Based on the second argument of Lemma \ref{lemma1}, we can obtain $\small \tau_{p+1}\geq1,\ \forall p\geq1$, then multiply $\small \tau_{p+1}^{\alpha-1}$ and $\small \tau_{p+1}$ to the left and right-hand side of (\ref{Most_Important_Ieq_1}), respectively, we have

\begin{small}
\begin{equation}\label{Most_Important_Ieq_2}
\begin{split}
\frac{2}{L}
\begin{bmatrix}
\tau_{p+1}^{\alpha-1}(\tau_{p+1}-1)\upsilon^{p}-\tau_{p+1}^{\alpha}\upsilon^{p+1}
\end{bmatrix}\geq
\|\tau_{p+1}(\mu^{p+1}-\zeta^{p+1})\|_{2}^{2}\\
+2\tau_{p+1}\langle\mu^{p+1}-\zeta^{p+1},\tau_{p+1}\zeta^{p+1}-(\tau_{p+1}-1)\mu^{p}-\mu^{*}\rangle.
\end{split}
\end{equation}
\end{small}Let $\small y_{1}=\tau_{p+1}\zeta^{p+1}$, $\small y_{2}=\tau_{p+1}\mu^{p+1}$ and $\small y_{3}=(\tau_{p+1}-1)\mu^{p}+\mu^{*}$, the right-hand side of (\ref{Most_Important_Ieq_2}) can be written as

\begin{small}
\begin{equation}
\|y_{2}-y_{1}\|_{2}^{2}+2\langle y_{2}-y_{1},y_{1}-y_{3}\rangle
=
\|y_{2}-y_{3}\|_{2}^{2}
-\|y_{1}-y_{3}\|_{2}^{2}.
\end{equation}
\end{small}Since $\small \tau_{p}^{\alpha}=\tau_{p+1}^{\alpha}-\tau_{p+1}^{\alpha-1}$, the inequality (\ref{Most_Important_Ieq_2}) is equivalent to

\begin{small}
\begin{equation}\label{Most_Important_Ieq_3}
\begin{split}
\frac{2}{L}
\begin{bmatrix}
\tau_{p}^{\alpha}\upsilon^{p}-\tau_{p+1}^{\alpha}\upsilon^{p+1}
\end{bmatrix}&\geq
\|y_{2}-y_{3}\|_{2}^{2}
-\|y_{1}-y_{3}\|_{2}^{2} \\
&=
\|\tau_{p+1}\mu^{p+1}-(\tau_{p+1}-1)\mu^{p}-\mu^{*}\|_{2}^{2}\\
&-
\|\tau_{p+1}\zeta^{p+1}-(\tau_{p+1}-1)\mu^{p}-\mu^{*}\|_{2}^{2}.
\end{split}
\end{equation}
\end{small}Let $\small \kappa_{p}=\tau_{p}\mu^{p}-(\tau_{p}-1)\mu^{p-1}-\mu^{*}$, combine with $\small \tau_{p+1}\zeta^{p+1}=\tau_{p+1}\mu^{p}+(\tau_{p}-1)(\mu^{p}-\mu^{p-1})$, the right-hand side of (\ref{Most_Important_Ieq_3}) is equal to $\small \|\kappa_{p+1}\|_{2}^{2}-\|\kappa_{p}\|_{2}^{2}$. Therefore, similar as Lemma 4.1 in \cite{Beck_2009}, we have the following conclusion

\begin{small}
\begin{equation}
\frac{2}{L}\tau_{p}^{\alpha}\upsilon^{p}
-\frac{2}{L}\tau_{p+1}^{\alpha}\upsilon^{p+1}\geq
\|\kappa_{p+1}\|_{2}^{2}-\|\kappa_{p}\|_{2}^{2}.
\end{equation}
\end{small}According to Lemma 4.2 in \cite{Beck_2009}, let $\small \bar{y}_{1}^{p}=\frac{2}{L}\tau_{p}^{\alpha}\upsilon^{p}$, $\small \bar{y}_{2}^{p}=\|\kappa_{p}\|_{2}^{2}$ and $\small \bar{y}_{3}=\|\mu^{0}-\mu^{*}\|_{2}^{2}$, we have $\small \bar{y}_{1}^{p}+\bar{y}_{2}^{p}\geq\bar{y}_{1}^{p+1}+\bar{y}_{2}^{p+1}$. Assume $\small \bar{y}_{1}^{1}+\bar{y}_{2}^{1}\leq\bar{y}_{3}$ holds, we have $\small \bar{y}_{1}^{p}+\bar{y}_{2}^{p}\leq\bar{y}_{3}$, which leads to $\small \bar{y}_{1}^{p}\leq\bar{y}_{3}$. Moreover, according to the second argument of Lemma \ref{lemma1}, we have

\begin{small}
\begin{equation*}
\begin{split}
\frac{2}{L}\tau_{p}^{\alpha}\upsilon^{p}\leq\|\mu^{0}-\mu^{*}\|_{2}^{2}
\Rightarrow
f(\mu^{p})-f(\mu^{*})
\leq
\frac{\alpha^{\alpha}L\|\mu^{0}-\mu^{*}\|_{2}^{2}}{2(p+\alpha-1)^{\alpha}}.
\end{split}
\end{equation*}
\end{small}The proof of the assumption $\small \bar{y}_{1}^{1}+\bar{y}_{2}^{1}\leq\bar{y}_{3}$ can be found in Theorem 4.4 of \cite{Beck_2009}. Then, according to the procedures in Theorem 3 of \cite{Giselsson_2013}, we conclude that

\begin{small}
\begin{equation}
\|\xi^{p}-\xi^{*}\|_{2}^{2}\leq
\frac{2}{\underline{\sigma}(\mathcal{H})}(f(\mu^{p})-f(\mu^{*}))
\leq
\frac{\alpha^{\alpha}L\|\mu^{0}-\mu^{*}\|_{2}^{2}}{\underline{\sigma}(\mathcal{H})(p+\alpha-1)^{\alpha}}.
\end{equation}
\end{small}In this way, the convergence rate (\ref{convergence_rate_general}) is obtained.
\end{proof}

\begin{algorithm}[t]
\caption{Look-up table generation for the $\alpha$th-order polynomial equation (\ref{alpha_polynomial}).}
\label{algorithm2}
\begin{algorithmic}[1]
\REQUIRE ~~\\
The order $\alpha\geq2$, the initial root $\tau_{1}=1$, the initial iteration index $p=1$ and the table length $\mathcal{P}$.
\ENSURE ~~Look-up table $\mathcal{T}_{\alpha}$. \\
\STATE Look-up table initialization: $\mathcal{T}_{\alpha}=[\tau_{1}]$.\\

\WHILE {$p\leq\mathcal{P}$}
\STATE Polynomial coefficients: $\boldsymbol{p_c}=[1,-1,\text{zeros}(1, \alpha-2),-\mathcal{T}_{\alpha}(\text{end})^{\alpha}]$.\\
\STATE Polynomial roots: $\boldsymbol{p_r}=\text{roots}(\boldsymbol{p_c})$.\\
\STATE Finding the positive real root $\tau_{p+1}$ in the vector $\boldsymbol{p_r}$.\\
\STATE Updating the look-up table: $\mathcal{T}_{\alpha} =[\mathcal{T}_{\alpha},\tau_{p+1}]$.\\
\STATE $\small p=p+1.$
\ENDWHILE
\end{algorithmic}
\end{algorithm}

\begin{figure}[hbtp]
\centering
\includegraphics[scale=0.53]{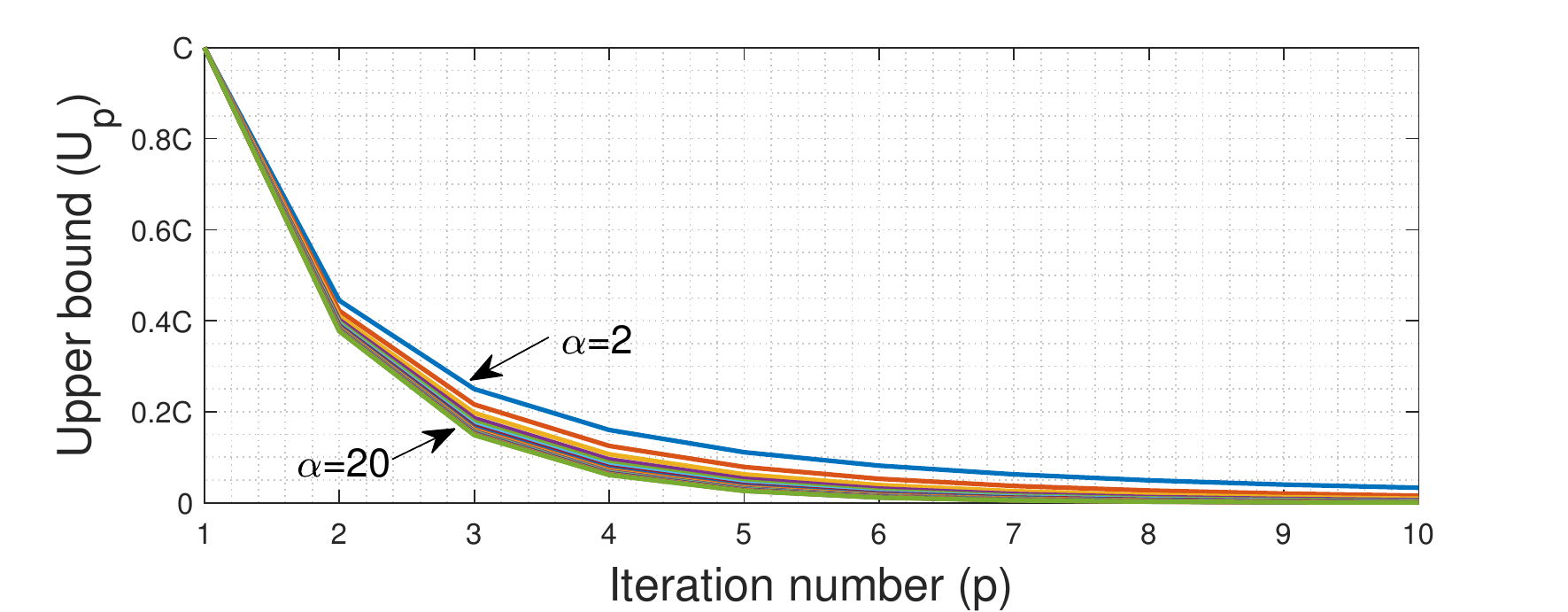}
\caption{Right-hand side of (\ref{convergence_rate_general}) with the variation of $\small \alpha$.}
\label{fig_upper_bound_rate}
\end{figure}

Theorem \ref{Theorem_2} shows that the FISTA in \cite{Beck_2009} is a special case of the proposed method and the iteration performance is determined by (\ref{alpha_polynomial}). Specifically, a suitable selection of the iterative parameter $\small \tau_{p}$ can improve the convergence rate, i.e., from $\small O(1/p^{2})$ in \cite{Beck_2009} to $\small O(1/p^{\alpha})$. To show the upper bound of the convergence rate can be reduced, denote the right-hand side of (\ref{convergence_rate_general}) as

\begin{small}
\begin{equation}\label{upper_bound}
Up
=
\frac{\alpha^{\alpha}}{(p+\alpha-1)^{\alpha}}C,\ p=1,2,\cdots
\end{equation}
\end{small}where $C$ is the constant part of the right-hand side of (\ref{convergence_rate_general}). The variation of $U_{p}$ with $p\in\{1,\cdots,10\}$ is shown in Fig. \ref{fig_upper_bound_rate}, in which different color lines denote different $\alpha\in\{2,\cdots,20\}$. Fig. \ref{fig_upper_bound_rate} implies that $U_{p}$ is decreasing with the increase of $\alpha$.

%\begin{figure}[hbtp]
%\centering
%\includegraphics[scale=0.53]{lower_bound_tau.eps}
%\caption{The value of $\tau_{p}$ and $(p+\alpha-1)/\alpha$ in the case of $\alpha=20$.}
%\label{fig_lower_bound_alpha}
%\end{figure}

%In this work, the look-up table is obtained by recursively solving the polynomial equation (\ref{alpha_polynomial}) in MATLAB environment, which can be summarized as Algorithm \ref{algorithm2}. Notice that the MATLAB function $\text{roots}(\cdot)$ is used for the polynomial root seeking.

%Based on which, we conclude that (\ref{property}) holds for (\ref{alpha_polynomial}) if $\small \tau_{1}=1$. For example, in the case of $\alpha=20$, the value of $\tau_{p}$ and $\frac{p+\alpha-1}{\alpha}$ for $p\in\{1,\cdots,2000\}$ are shown in Fig. \ref{fig_lower_bound_alpha}. In other cases of $\alpha$, the similar conclusions can be obtained.

\subsection{Cholesky Decomposition of $\mathcal{H}$}

According to the QP in Section \ref{section2}, the quadratic objective term $\mathcal{H}$ in (\ref{standard_QP}) may be a dense matrix, then more computation time could be consumed than a banded matrix if solving (\ref{standard_QP}) by Algorithm \ref{algorithm1} directly. To cope with this difficulty, the matrix decomposition technique can be used. Since $\mathcal{H}$ is symmetric and positive definite, there exists the Cholesky decomposition $\mathcal{H}=\mathcal{Z}^{T}\mathcal{Z}$, based on which, the quadratic programming problem (\ref{standard_QP}) can be formulated into

\begin{small}
\begin{equation}\label{compact_original_problem3}
\begin{split}
&\min\limits_{\psi}
\frac{1}{2}
\psi^{T}I\psi
+
\mathcal{G}^{T}\mathcal{Z}^{-1}\psi
\\
&\ s.t.\ \mathcal{A}\mathcal{Z}^{-1}\psi\leq\mathcal{B},
\end{split}
\end{equation}
\end{small}where $\psi=\mathcal{Z}\xi$. Since $\mathcal{Z}$ is a upper triangular matrix with real and positive diagonal components, (\ref{compact_original_problem3}) can be solved by Algorithm \ref{algorithm1} and the control input can be calculated by $\xi=\mathcal{Z}^{-1}\psi$. In this way, the quadratic objective term is transformed into the identity matrix, which can reduce the computation time in step $4$ of Algorithm \ref{algorithm1}.

\begin{table*}[t]
\centering
\caption{Iteration performance with four methods.}
%\begin{tabular}{@{}lcccccccccccc@{}}
\begin{tabular}{lcccccccc}
\toprule
& \multicolumn{2}{c}{{\ul \ \ \ \ \ \ \ \ \ \ \ \ n=m=2\ \ \ \ \ \ \ \ \ \ \ }}
& \multicolumn{2}{c}{{\ul \ \ \ \ \ \ \ \ \ \ \ \ \ n=m=4\ \ \ \ \ \ \ \ \ \ \ }}
& \multicolumn{2}{c}{{\ul \ \ \ \ \ \ \ \ \ \ \ \ \ n=m=6\ \ \ \ \ \ \ \ \ \ \ }}
& \multicolumn{2}{c}{{\ul \ \ \ \ \ \ \ \ \ \ \ \ \ n=m=8\ \ \ \ \ \ \ \ \ \ }}
\\
& \multicolumn{2}{c}{{\ul \ \ \ \ \ \ vars/cons: 10/40\ \ \ \ \ \ }}
& \multicolumn{2}{c}{{\ul \ \ \ \ \ \ vars/cons: 20/80\ \ \ \ \ \ }}
& \multicolumn{2}{c}{{\ul \ \ \ \ \ \ vars/cons: 30/120\ \ \ \ \ \ }}
& \multicolumn{2}{c}{{\ul \ \ \ \ \ \ vars/cons: 40/160\ \ \ \ \ \ }}
\\
& \multicolumn{1}{l}{\ \ ave.iter}
& \multicolumn{1}{l}{\ \ ave.time (s)}
& \multicolumn{1}{l}{\ \ ave.iter}
& \multicolumn{1}{l}{\ \ ave.time (s)}
& \multicolumn{1}{l}{\ \ ave.iter}
& \multicolumn{1}{l}{\ \ ave.time (s)}
& \multicolumn{1}{l}{\ \ ave.iter}
& \multicolumn{1}{l}{\ \ ave.time (s)} \\
\midrule
MOSEK
& --                                     & 0.10149                           & --                                     & 0.10226                            & --                                     & 0.10873                            & --                                     & 0.10887                            \\
ECOS
& --                                     & 0.00452                            & --                                     & 0.00659                            & --                                     & 0.00849                            & --                                     & 0.01287                            \\
FISTA
& 29.37                                     & 0.00098                            & 115.95                                     & 0.00398                            & 159.76                                     & 0.00800                            & 272.88                                     & 0.01777                            \\
Algorithm \ref{algorithm1} ($\alpha=20$)
& 26.56                                     & 0.00078                            & 78.15                                     & 0.00251                            & 119.03                                     & 0.00484                            & 176.00                                     & 0.00785                            \\
\bottomrule
\end{tabular}\label{table2}
\end{table*}

\section{Performance Analysis based on MPC}

\label{section4}

\subsection{Formulation of Standard MPC}

Consider the discrete-time linear system as

\begin{small}% 状态方程
\begin{equation}\label{state_equation}
x_{k+1}=Ax_{k}+Bu_{k},
\end{equation}
\end{small}where $A$ and $B$ are known time-invariant matrixes. $x_{k}\in\mathcal{R}^{n}$ and $u_{k}\in\mathcal{R}^{m}$ have linear constraints as $Fx_{k}\leq\boldsymbol{1}$ and $Gu_{k}\leq\boldsymbol{1}$, respectively, in which $F\in\mathcal{R}^{f\times n}$, $G\in\mathcal{R}^{g\times m}$ and $\boldsymbol{1}$ is a vector with each component is equal to $1$. The standard MPC problem can be presented as

\begin{small}
\begin{equation}\label{original_MPC}
\min \limits_{\boldsymbol{u}_{k}}
J(x_{k},\ \boldsymbol{u}_{k}),\ \
\text{s.t.}\ \boldsymbol{u}_{k}\in \mathbb{U},
\end{equation}
\end{small}where $x_{k}$ is the current state, the decision variables is the nominal input trajectory $\boldsymbol{u}_{k}=(u_{0|k},\cdots,u_{N-1|k})\in\mathcal{R}^{Nm}$, $N$ is the prediction horizon. The construction of $\mathbb{U}$ can be found in \cite{Mayne_2000}. Moreover, the cost function $J(x_{k},\boldsymbol{u}_{k})$ is

\begin{small}
\begin{equation}\label{original_cost_function}
J(x_{k},\boldsymbol{u}_{k})
=
\frac{1}{2}
\sum_{l=0}^{N-1}
\begin{bmatrix}
\|x_{l|k}\|_{Q}^{2}
+
\|u_{l|k}\|_{R}^{2}
\end{bmatrix}
+
\frac{1}{2}
\|x_{N|k}\|_{P}^{2},
\end{equation}
\end{small}where $l|k$ denotes the $l$-th step ahead prediction from the current time $k$. $Q$, $R$ and $P$ are positive definite matrices. $P$ is chosen as the solution of the discrete algebraic Riccati equation of the unconstrained problem. The standard MPC problem (\ref{original_MPC}) can be formulated as the QP problem (\ref{standard_QP}), which has been shown in Appendix \ref{appendices}.

\subsection{Existing Methods for Comparison}

The performance comparisons with the optimization software MOSEK \cite{Erling2003}, the embedded solver ECOS \cite{Domahidi2013ecos} and the FISTA \cite{Beck_2009} have been provided. The MOSEK and ECOS quadratic programming functions in MATLAB environment, i.e., $\text{mskqpopt}(\cdot)$ and $\text{ecosqp}(\cdot)$, are used, they are invoked as

\begin{small}
\begin{subequations}
\begin{align}
\begin{split}
&[\text{sol}]
=
\text{mskqpopt}
(\mathcal{H},
\mathcal{G},
\mathcal{A},
[\ ],
\mathcal{B},
[\ ],
[\ ],
[\ ],
\text{'minimize info');}\\
&\text{time}
=
\text{sol.info.MSK\_DINF\_INTPNT\_TIME;}
\end{split}
\\
\nonumber
\\
&[\text{sol},
\sim,
\sim,
\sim,
\sim,
\text{time}]
=
\text{ecosqp}
(\mathcal{H},
\mathcal{G},
\mathcal{A},
\mathcal{B});
\end{align}
\end{subequations}
\end{small}The version of MOSEK is 9.2.43 and the numerical experiments are proceeded by running MATLAB R2018a on Windows 10 platform with 2.9G Core i5 processor and 8GB RAM.

\subsection{Performance Evaluation of Algorithm \ref{algorithm1}}

Four kinds of system scales are considered, they are $n=m=2,\ 4,\ 6,\ 8$. The performance of above methods are evaluated by solving $400$ random MPC problems in each system scale. Since we develop the efficient solving method in one control step, without loss of generality, a batch of stable and controllable plants with the random initial conditions and constraints are used. The components in the dynamics and input matrices are randomly selected from the interval $[-1,1]$. Each component in the state and input are upper and lower bounded by random bounds generated from intervals $[1,10]$ and $[-10,-1]$ respectively. The prediction horizon is $\small N=5$, the controller parameters are $\small Q=I$ and $\small R=10I$. Only the iteration process in the first control step is considered and the stop criterion is $\|\xi^{p}-\xi^{p-1}\|_{2}\leq10^{-3}$. Let $\alpha=20$ in Algorithm \ref{algorithm1}, the results are shown in Table \ref{table2}, in which "ave.iter" and "ave.time" are the abbreviations of "average iteration number" and "average execution time", and "vars/cons" denotes the number of variables and constraints. Table \ref{table2} implies that the average execution time can be reduced by using the proposed method. Noticing that Table \ref{table2} shows that the execution time of Algorithm \ref{algorithm1} and ECOS are much faster than MOSEK, hence, only the discussions about Algorithm \ref{algorithm1} and ECOS are provided in the rest of the letter for the purpose of conciseness.

\begin{figure}[hbtp]
\centering
\includegraphics[scale=0.53]{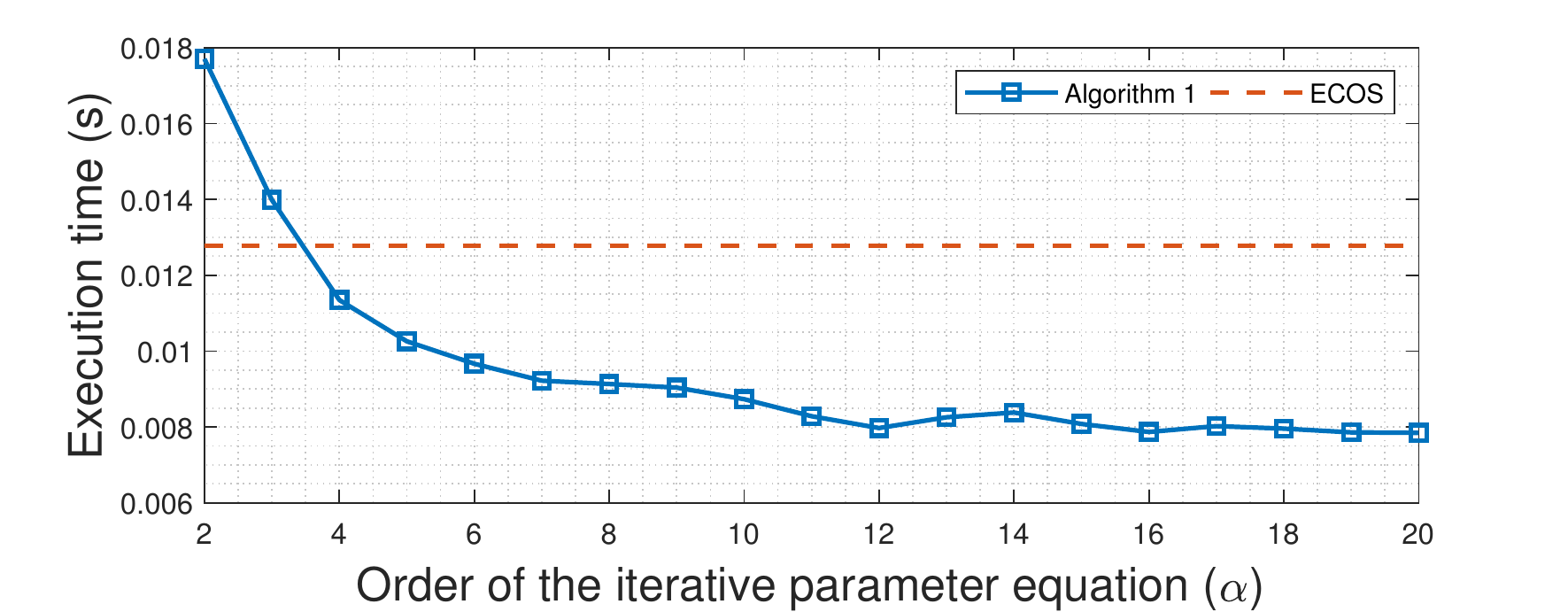}
\caption{Average execution time of Algorithm \ref{algorithm1} and ECOS in the case of $\small n=m=8$.}
\label{fig_exc_time}
\end{figure}

To show the performance improvement of Algorithm \ref{algorithm1} with the increase of $\small \alpha\in\{2,\cdots,20\}$, an example in the case of $\small n=m=8$ is given in Fig. \ref{fig_exc_time}, which presents the results in terms of the average execution time. Since only the upper bound of convergence rate is reduced by increasing $\small\alpha$, the execution time may not strictly decline. Fig. \ref{fig_exc_time} implies that the execution time of Algorithm \ref{algorithm1} can be shorten by increasing $\small\alpha$ and faster than the ECOS for solving the same MPC optimization problem. Notice that there is no significant difference in the execution time if $\small\alpha$ keeps increasing. In fact, it depends on the stop criterion, therefore, a suitable $\small\alpha$ can be selected according to the required solution accuracy.

\begin{figure}[hbtp]
\centering
\includegraphics[scale=0.54]{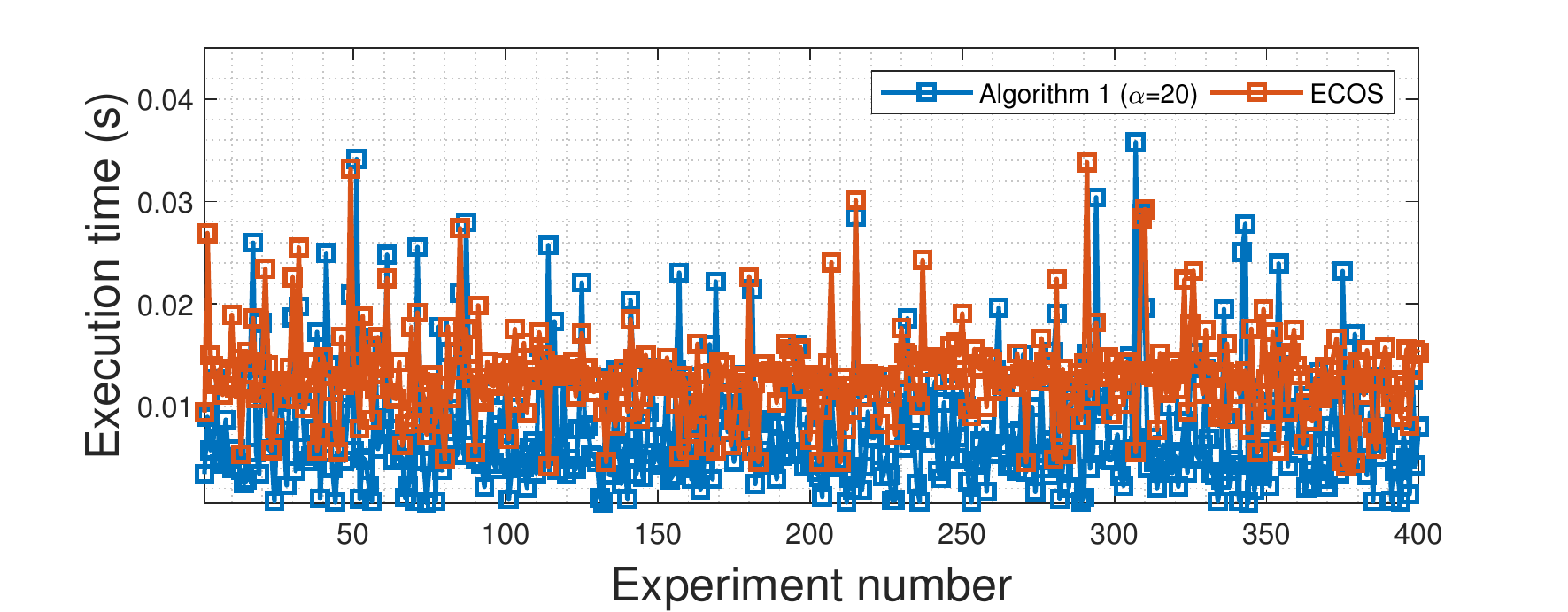}
\caption{Execution time for each experiment in the case of $n=m=8$.}
\label{fig_exc_time_8}
\end{figure}

\subsection{Statistical Significance of Experimental Result}

Table \ref{table2} verifies the effectiveness of Algorithm \ref{algorithm1} by using the average execution time, the statistical significance is discussed as follows. Since the sample size is large in our test, i.e., $400$ random experiments in each case, the paired $t$-test developed in Section $10.3$ and $12.3$ of \cite{Wackerly2008} can be used. Denote the average execution time under the ECOS and Algorithm \ref{algorithm1} as $\mu_e$ and $\mu_a$, and the difference of execution time between the two methods as $D_{i}$ for $i=1,\cdots,M$, in which $M=400$. If the average execution time for the ECOS is larger, then $\mu_{D}=\mu_{e}-\mu_{a}>0$. Thus, we test

\begin{small}
\begin{equation*}
\text{H}_{0}:\ \mu_{D}=0\ \
\text{versus}\ \
\text{H}_{1}:\ \mu_{D}>0.
\end{equation*}
\end{small}Define the sample mean and variance as

\begin{small}
\begin{equation*}
\bar{D}
=
\frac{1}{M}
\sum_{i=1}^{M}D_{i},\ \
S_{D}^{2}
=
\frac{1}{M-1}
\sum_{i=1}^{M}
(D_{i}-\bar{D})^{2},
\end{equation*}
\end{small}then the test statistic is calculated as

\begin{small}
\begin{equation*}
t=
\frac{\bar{D}-\mu_{D}}{S_{D}/\sqrt{M}},
\end{equation*}
\end{small}which is the observed value of the statistic under the null hypothesis $\text{H}_{0}$. In the case of $n=m=8$, for example, the execution time for each random experiment is given in Fig. \ref{fig_exc_time_8} and the test statistic is $t=15.7623$, which leads to an extremely small $p$-value compared with the significance level $0.001$. Hence, the result is statistically significant to suggest that the ECOS yields a larger execution time than does Algorithm \ref{algorithm1}. In other cases of the system scale, the similar results can be obtained.

\begin{figure}[hbtp]
\centering
\includegraphics[scale=0.53]{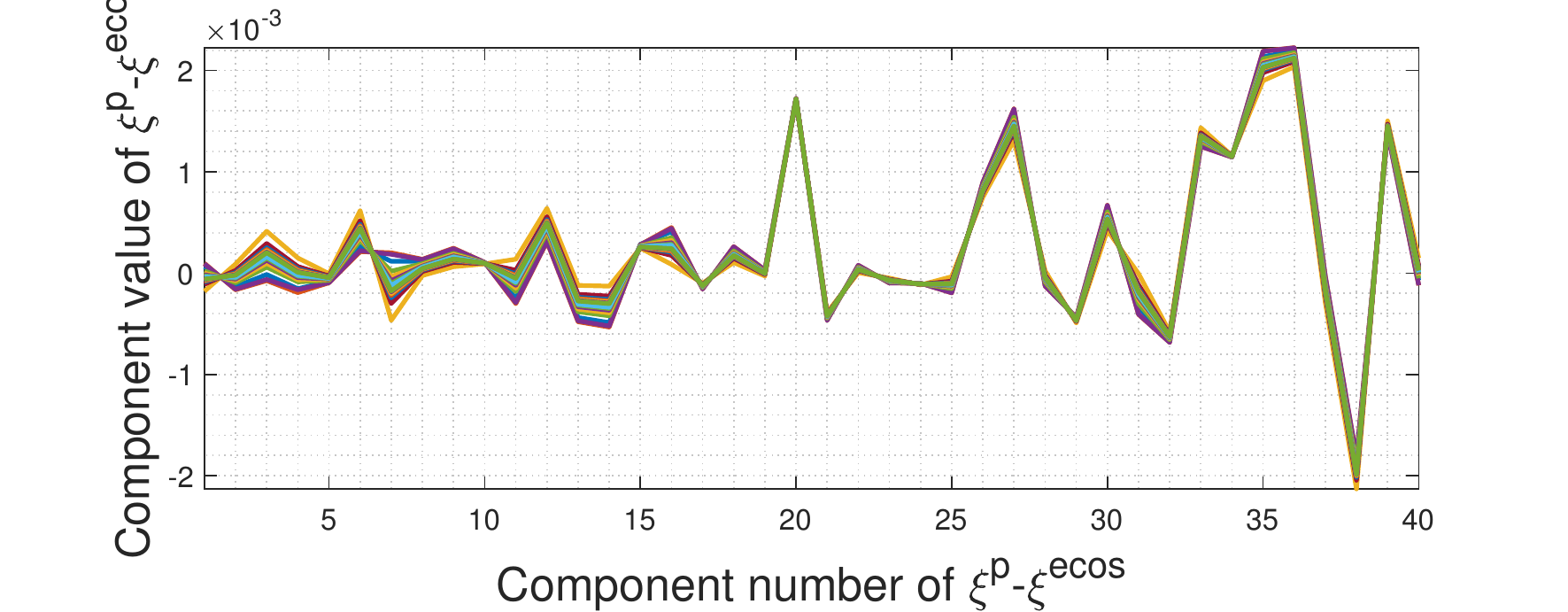}
\caption{Solution error between Algorithm \ref{algorithm1} and ECOS in the case of $n=m=8$.}
\label{fig_solution_error}
\end{figure}

\begin{figure}[hbtp]
\centering
\includegraphics[scale=0.53]{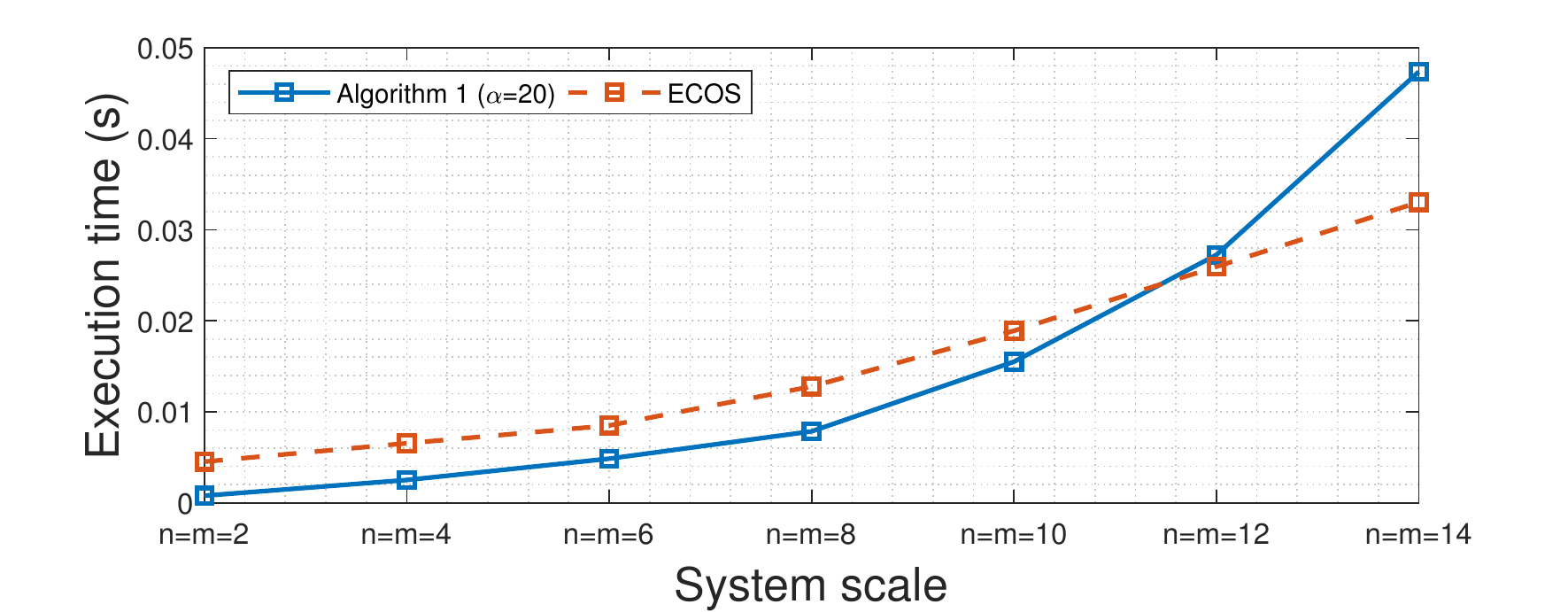}
\caption{Average execution time of Algorithm \ref{algorithm1} and ECOS at different system scales.}
\label{fig_limit_analysis}
\end{figure}

\subsection{Error and Limitation Analysis of Algorithm \ref{algorithm1}}

To verify the accuracy of the solutions of Algorithm \ref{algorithm1}, the solution error $\small \xi^{p}-\xi^{ecos}$ is calculated as $\small \xi^{p}$ satisfies the stop criterion, in which the ECOS solution is denoted as $\small \xi^{ecos}$. For example, give one random MPC problem in the case of $n=m=8$, each component of solution error is shown in Fig. \ref{fig_solution_error}, in which different color lines denote different $\small\alpha$. The results in Fig. \ref{fig_solution_error} reveal that the component is not greater than $\small2.2\times10^{-3}$ in each case of $\small\alpha$, hence, the solution of Algorithm \ref{algorithm1} is close to the ECOS solution. Moreover, notice that the solution error with different $\small\alpha$ is close to each other, which means that the selection of $\small\alpha$ has little influence on the final solution. In other random optimization problems, the same conclusion can be obtained. In this way, the accuracy of the solutions of Algorithm \ref{algorithm1} is verified. However, the limitation of Algorithm \ref{algorithm1} is that it is only suitable for the small size MPC problems. The illustration is given as Fig. \ref{fig_limit_analysis}, in which the average execution time of Algorithm \ref{algorithm1} ($\alpha=20$) and ECOS are presented. Fig. \ref{fig_limit_analysis} implies that the performance of Algorithm \ref{algorithm1} degrades with the increase of the system scale. The extension of Algorithm \ref{algorithm1} such that the large-scale optimization problems can be solved efficiently is the topic of the future research.

\section{Conclusion}

\label{section5}

In this letter, QP problems are solved by a novel PGM. We show that the FISTA is a special case of the proposed method and the convergence rate can be improved from $O(1/p^{2})$ to $O(1/p^{\alpha})$ by selecting the positive real roots of a group of high order polynomial equations as the iterative parameters. Based on a batch of random experiments, the effectiveness of the proposed method on MPC problem has been verified.

\begin{appendices}

\section{From Standard MPC to QP}

\label{appendices}

According to the nominal model (\ref{state_equation}), the relationship between the predicted nominal states and inputs in a finite horizon $\small N$ can be expressed as

\begin{small}
\begin{equation}
\boldsymbol{x}_{k}
=
A_{1}x_{k}
+
A_{2}\boldsymbol{u}_{k},
\end{equation}
\end{small}where

\begin{small}
\begin{equation}
A_{1}
=
\begin{bmatrix}
  A \\
  \vdots \\
  A^{N}
\end{bmatrix},\
A_{2}
=
\begin{bmatrix}
  B & \boldsymbol{0} & \cdots & \boldsymbol{0} \\
  AB & B & \cdots & \boldsymbol{0} \\
  \vdots & \vdots & \ddots & \boldsymbol{0} \\
  A^{N-1}B & A^{N-2}B & \cdots & B
\end{bmatrix}.
\end{equation}
\end{small}Denote $\small Q_{1}=\text{diag}(Q,\cdots,Q,P)\in\mathcal{R}^{Nn\times Nn}$ and $\small R_{1}=\text{diag}(R,\cdots,R)\in\mathcal{R}^{Nm\times Nm}$, the objective (\ref{original_cost_function}) containing the equality constraints can be written as

\begin{small}
\begin{equation}
J
(\boldsymbol{x}_{k},\boldsymbol{u}_{k})
=
\frac{1}{2}
\boldsymbol{u}_{k}^{T}\mathcal{H}\boldsymbol{u}_{k}
+
\mathcal{G}(x_{k})^{T}\boldsymbol{u}_{k}
+
c(x_{k}),
\end{equation}
\end{small}where $\mathcal{H}=A_{2}^{T}Q_{1}A_{2}+R_{1}$, $\mathcal{G}(x_{k})=A_{2}^{T}Q_{1}A_{1}x_{k}$ and $c(x_{k})=\frac{1}{2}x_{k}^{T}A_{1}^{T}Q_{1}A_{1}x_{k}$. Then the standard quadratic optimization objective is obtained. Let $\tilde{F}=\text{diag}(F,\cdots,F)\in\mathcal{R}^{Nf\times Nn}$, $\tilde{\Phi}=(\boldsymbol{0},\Phi)\in\mathcal{R}^{w\times Nn}$ ($\Phi$ is the terminal constraint on the predicted state $x_{N|k}$), $\bar{F}=(\tilde{F}^{T},\tilde{\Phi}^{T})^{T}\in\mathcal{R}^{(Nf+w)\times Nn}$ and $\bar{G}=\text{diag}(G,\cdots,G)\in\mathcal{R}^{Ng\times Nm}$, the linear constraints of (\ref{original_MPC}) can be written as

\begin{small}
\begin{equation}
\mathcal{A}
\boldsymbol{u}_{k}
\leq
\mathcal{B}(x_{k}),
\end{equation}
\end{small}where

\begin{small}
\begin{equation}
\mathcal{A}
=
\begin{bmatrix}
  \bar{F}A_{2} \\
  \bar{G}
\end{bmatrix},\
\mathcal{B}(x_{k})
=
\begin{bmatrix}
  \boldsymbol{1}-\bar{F}A_{1}x_{k} \\
  \boldsymbol{1}
\end{bmatrix}.
\end{equation}
\end{small}In this way, the MPC problem (\ref{original_MPC}) is formulated into the quadratic programming form (\ref{standard_QP}). After solving the MPC problem, the first term of the optimal input trajectory $\small \boldsymbol{u}_{k}^{*}$ is imposed to the plant at time $\small k$.

\end{appendices}

\bibliographystyle{IEEEtranS}
\bibliography{ieeepesp}

% Generated by IEEEtranS.bst, version: 1.14 (2015/08/26)
\begin{thebibliography}{10}
\providecommand{\url}[1]{#1}
\csname url@samestyle\endcsname
\providecommand{\newblock}{\relax}
\providecommand{\bibinfo}[2]{#2}
\providecommand{\BIBentrySTDinterwordspacing}{\spaceskip=0pt\relax}
\providecommand{\BIBentryALTinterwordstretchfactor}{4}
\providecommand{\BIBentryALTinterwordspacing}{\spaceskip=\fontdimen2\font plus
\BIBentryALTinterwordstretchfactor\fontdimen3\font minus
  \fontdimen4\font\relax}
\providecommand{\BIBforeignlanguage}[2]{{%
\expandafter\ifx\csname l@#1\endcsname\relax
\typeout{** WARNING: IEEEtranS.bst: No hyphenation pattern has been}%
\typeout{** loaded for the language `#1'. Using the pattern for}%
\typeout{** the default language instead.}%
\else
\language=\csname l@#1\endcsname
\fi
#2}}
\providecommand{\BIBdecl}{\relax}
\BIBdecl

\bibitem{Erling2003}
E.~D. Andersen, C.~Roos, and T.~Terlaky, ``On implementing a primal-dual
  interior-point method for conic quadratic optimization,'' \emph{Mathematical
  Programming}, vol.~95, pp. 249--277, 2003.

\bibitem{Alberto_2021}
D.~Arnström, A.~Bemporad, and D.~Axehill, ``Complexity certification of
  proximal-point methods for numerically stable quadratic programming,''
  \emph{IEEE Control Systems Letters}, vol.~5, no.~4, pp. 1381--1386, 2021.

\bibitem{Beck_2009}
A.~Beck and M.~Teboulle, ``A fast iterative shrinkage-thresholding algorithm
  for linear inverse problems,'' \emph{SIAM Journal on Imaging Sciences},
  vol.~2, no.~1, p. 183–202, 2009.

\bibitem{Boyd}
S.~Boyd and L.~Vandenberghe, \emph{Convex optimization}.\hskip 1em plus 0.5em
  minus 0.4em\relax New York, NY: Cambridge University Press, 2004.

\bibitem{Domahidi2013ecos}
A.~Domahidi, E.~Chu, and S.~Boyd, ``{ECOS}: {A}n {SOCP} solver for embedded
  systems,'' in \emph{European Control Conference (ECC)}, 2013, pp. 3071--3076.

\bibitem{Domahidi2012}
A.~Domahidi, A.~U. Zgraggen, M.~N. Zeilinger, and et~al, ``Efficient interior
  point methods for multistage problems arising in receding horizon control,''
  in \emph{Conference on Decision and Control (CDC)}, 2012, pp. 668--674.

\bibitem{Ferreau_2014}
H.~Ferreau, C.~Kirches, A.~Potschka, and et~al, ``qp{OASES}: a parametric
  active-set algorithm for quadratic programming,'' \emph{Mathematical
  programming computation}, vol.~6, no.~4, pp. 327--363, 2014.

\bibitem{Frasch_2015}
J.~V. Frasch, S.~Sager, and M.~Diehl, ``A parallel quadratic programming method
  for dynamic optimization problems,'' \emph{Mathematical programming
  computation}, vol.~7, no.~3, pp. 289--329, 2015.

\bibitem{Giselsson2014}
P.~Giselsson, ``Improved fast dual gradient methods for embedded model
  predictive control,'' in \emph{IFAC world congress}, 2014, pp. 2303--2309.

\bibitem{Giselsson2015}
P.~Giselsson and S.~Boyd, ``Metric selection in fast dual forward–backward
  splitting,'' \emph{Automatica}, vol.~62, pp. 1--10, 2015.

\bibitem{Giselsson_2013}
P.~Giselsson, M.~D. Doan, T.~Keviczky, and et~al, ``Accelerated gradient
  methods and dual decomposition in distributed model predictive control,''
  \emph{Automatica}, vol.~49, p. 829–833, 2013.

\bibitem{Juan_2014}
J.~L. Jerez, P.~J. Goulart, S.~Richter, and et~al, ``Embedded online
  optimization for model predictive control at megahertz rates,'' \emph{IEEE
  Transactions on Automatic Control}, vol.~59, no.~12, pp. 3238--3251, 2014.

\bibitem{Mayne_2000}
D.~Q. Mayne, J.~B. Rawlings, C.~V. Rao, and P.~O.~M. Scokaert, ``Constrained
  model predictive control: Stability and optimality,'' \emph{Automatica},
  vol.~36, pp. 789--814, 2000.

\bibitem{Nesterov2013}
Y.~Nesterov, ``Gradient methods for minimizing composite functions,''
  \emph{Mathematical Programming}, vol. 140, pp. 125--161, 2013.

\bibitem{Parys_2019}
R.~V. Parys, M.~Verbandt, J.~Swevers, and G.~Pipeleers, ``Real-time proximal
  gradient method for embedded linear {MPC},'' \emph{Mechatronics}, vol.~59,
  pp. 1--9, 2019.

\bibitem{Qin_2003}
S.~J. Qin and T.~A. Badgwell, ``A survey of industrial model predictive control
  technology,'' \emph{Control Engineering Practice}, vol.~11, no.~7, pp.
  733--764, 2003.

\bibitem{Wackerly2008}
D.~D. Wackerly, W.~Mendenhall, and R.~L. Scheaffer, \emph{Mathematical
  statistics with applications, Seventh Edition}.\hskip 1em plus 0.5em minus
  0.4em\relax Belmont: Thomson Higher Education, 2008.

\end{thebibliography}

\end{document}